\documentclass[sigconf]{acmart}

\setcopyright{acmcopyright}
\copyrightyear{2024}
\acmYear{2024}

\acmConference[]{}{}{}
\acmBooktitle{}
\acmPrice{}
\acmISBN{}

\usepackage{amsthm,amsmath}
\usepackage{tikz}
\usepackage{xfrac}

\usepackage{mathtools}
\usepackage{enumerate}
\newcommand{\cF} { {\mathcal{F}}}
\newcommand{\cR} { {\mathcal{R}}}
\newcommand{\cZ} { {\mathcal{Z}}}
\newcommand{\bfu} { {\mathbf{u}}}
\newcommand{\bfa} { {\mathbf{a}}}
\newcommand{\bfc} { {\mathbf{c}}}
\newcommand{\bfm} { {\mathbf{m}}}

\newcommand{\bZ} { {\mathbb{Z}}}

\newcommand{\GL} { {\rm GL}}
\newcommand{\trdeg} { {\rm tr.deg}}

\newtheorem{theorem}{Theorem}

\newtheorem{thm}[theorem]{Theorem}
\newtheorem{corollary}[theorem]{Corollary}
\newtheorem{lemma}[theorem]{Lemma}
\newtheorem{remark}[theorem]{Remark}
\newtheorem{algorithm}[theorem]{Algorithm}
\newtheorem{problem}[theorem]{Problem}
\newtheorem{defi}[theorem]{Definition}
\newtheorem{example}[theorem]{Example}

\newtheorem{hypo}[theorem]{Hypothesis}

\makeatletter
\newcommand{\myitem}[1]{%
\item[(#1)]\protected@edef\@currentlabel{#1}%
}
\makeatother

\def\eatspace#1{#1}
\def\step#1#2{\par\kern1pt\hangindent#2em\hangafter=1\noindent\rlap{\small#1}\kern#2em\relax\eatspace}

\let\set\mathbb
\def\<#1>{\langle#1\rangle}

\def\lcm{\operatorname{lcm}}

\def\id{\operatorname{id}}

\def\disp{\operatorname{disp}}
\def\diag{\operatorname{diag}}

\def\si{\sigma}

\begin{document}
\fancyhead{}
\title{Parallel Summation in P-Recursive Extensions}
\thanks{S.\ Chen was partially supported by the National Key R\&D Program of China (No. 2023YFA1009401),
the NSFC grant (No. 12271511), CAS Project for Young Scientists in Basic Research (Grant
No. YSBR-034), and the CAS Fund of the Youth Innovation Promotion Association (No. Y2022001). R. Feng was partially supported by the National Key R\&D Program of China (No. 2023YFA1009401) and the National Key Research and Development Project 2020YFA0712300.
M.\ Kauers was supported by the Austrian FWF grants PAT 9952223 and I6130-N.
X.\ Li was partially supported by the Land Oberösterreich through the LIT-AI Lab
}

\author[S.\ Chen, R.\ Feng, M.\ Kauers, X.\ Li]{Shaoshi Chen$^{a, b}$, Ruyong Feng$^{a, b}$, Manuel Kauers$^c$, Xiuyun Li$^{a, b, c}$}

\affiliation{%
  \institution{$^a$KLMM, Academy of Mathematics and Systems Science, Chinese Academy of Sciences, Beijing 100190, China}
  \institution{$^b$School of Mathematical Sciences, University of Chinese Academy of Sciences, Beijing 100049, China}
  \institution{$^c$Institute for Algebra, Johannes Kepler University, Linz, A4040,  Austria}
  \state{}
  \postcode{}
  \country{}
}
\email{schen@amss.ac.cn,ryfeng@amss.ac.cn,manuel.kauers@jku.at,lixiuyun@amss.ac.cn}

\begin{abstract}
  We propose investigating a summation analog of the paradigm for parallel integration. 
  We make some first steps towards an indefinite summation method applicable to summands
  that rationally depend on the summation index and a
  P-recursive sequence and its shifts. There is a distinction between so-called normal and so-called special polynomials. Under the assumption that the
  corresponding difference field has no unnatural constants, we are able to
  predict the normal polynomials appearing in the denominator of a potential closed
  form. We can also handle the numerator. Our method is incomplete so far
  as we cannot predict the special polynomials appearing in the denominator. However, we do have
  some structural results about special polynomials for the setting under
  consideration.
\end{abstract}


\maketitle

\section{Introduction}

The main difference between the first and the second edition of Manuel Bronstein's classical
textbook on symbolic integration~\cite{BronsteinBook} is an additional tenth chapter about parallel
integration, which is based on his last paper~\cite{Bronstein2007} on the subject.
Parallel integration is an alternative approach to the more widely known Risch algorithm
for indefinite integration, whose careful
description dominates the remainder of Bronstein's book.
Parallel integration is also known as the
Risch-Norman algorithm~\cite{Norman1977, Harrington1979, Fitch1981, Davenport1982a, Davenport1982b, DavenportTrager1985, GeddesStefanus1989} and
as poorman's integrator~\cite{pmint}.

Although the technique is not complete, i.e., it fails to find a closed form of certain
integrals, it is an attractive alternative to a full implementation of the Risch algorithm,
which is guaranteed to find a closed form whenever there is one. One advantage is that it
is much easier to program. Indeed, Bronstein's Maple implementation~\cite{pmint} barely needs
100 lines of code. A second advantage is that it extends more easily to integrals of
non-elementary functions. For example, it can find the evaluation
\[
  \int\frac{x^2+(x^2+2)W(x^2)}{x(1+W(x^2))^2}dx=\frac12\frac{x^2}{W(x^2)}+\log(1+W(x^2))
\]
involving the Lambert $W$ function~\cite{LambertW}. This is not only interesting because $W$
is defined by a nonlinear equation, but also because there is a factor in the denominator of
the closed form that is not already present in the integrand.

In a seminar talk that never led to a formal publication, Zimmermann observed that parallel integration can be combined with the
concept of creative telescoping~\cite{Zeilberger1991} in order to handle definite integrals
involving a parameter, similar as done by Raab~\cite{Raab2013} with Risch's algorithm.
A version of parallel integration for integrals involving algebraic
functions was presented by B\"ottner in~\cite{Boettner2010} and for integrals of Airy functions by Du and Raab in~\cite{DuRaab2023}.

To our knowledge, the idea of parallel integration has not yet been translated to the
setting of symbolic summation. The goal of the present paper is to do so. A summation example
that is similar to the above integral can be given in terms of the logistic sequence $t_n$~\cite[Example 1.9, Chapter 1]{Elaydi2005}
satisfying the nonlinear recurrence equation $t_{n+1} = t_n (1-t_n)$ with $t_0 \in (0, 1)$. Here
we have the summation identity
\[
\sum_{k=0}^{n-1}\frac{1}{1-t_k} = \sum_{k=0}^{n-1} \left(\frac{1}{t_{k+1}} - \frac{1}{t_{k}}\right) = \frac{1}{t_{n}} - \frac{1}{t_0},
\]
and again, the denominator of the closed form contains a factor that is not already present in the
summand. 

On the other hand, the denominator of a closed form is not completely unpredictable.
Like in parallel integration, we can distinguish the \emph{special} and the \emph{normal}
part of a denominator. Based on this distinction, in Section~\ref{SECT:general} we show 
how the normal part of the denominator of a closed form depends on the normal part of
the denominator of the corresponding summand. Unfortunately, we do not have a complete
understanding of the special part, but we do have some results that limit the number of
special polynomials (Sect.~\ref{sec:numberirredspecial}). More can be said if we focus
on a more specific setting.

Sect.~\ref{SECT:general} is about the general paradigm of parallel summation, which
in principle could be applied to many different specific settings. In Sect.~\ref{sec:p-finite} we
restrict the attention to one such setting. We consider summation problems of the form
\[
  \sum_{k=0}^n\operatorname{rat}(k,f(k),f(k+1),\dots,f(k+r-1)),
\]
where $\operatorname{rat}$ is a multivariate rational function and $f$ is defined by a
linear recurrence of order~$r$ with polynomial coefficients. The task is to decide
whether a given sum of this type can be written as a rational function in $n,f(n),\dots,f(n+r-1)$.
While we are not (yet) able to solve this task in full generality, the idea of parallel
summation provides a significant step towards such an algorithm. We can describe more
precisely the structure of special polynomials in this case (Sect.~\ref{SUBSECT:structureofspecialpolynomials}), and we can
effectively solve the $\sigma$-equivalence problem (Sect.~\ref{SUBSECT:equivalence}), which implies that
we can completely identify the normal part of the denominator of any closed form. Similar
as in the differential case~\cite{BronsteinBook, Bronstein2007}, the special part has to
be determined heuristically, unless we impose further restrictions on the setting (Sect.~\ref{sec:cfinite}).

\section{Parallel summation}\label{SECT:general}

Similar to parallel integration, the general idea of parallel summation is to avoid the recursive nature of
summation algorithms such as Karr's algorithm by viewing the summand as an element of a field of multivariate rational functions
over a ground field. We now set up the general algebraic foundation for parallel summation and list some related problems.
Like in the situation of parallel integration, these problems are in general far from being solved.

Let $A$ be a ring and $\sigma\colon A \rightarrow A$ be an automorphism of~$A$. We call the pair $(A, \sigma)$
a \emph{difference ring} and a \emph{difference field} if $A$ is a field. Note that the set $\{a\in A \mid \sigma(a)=a\}$ forms a subring of $A$ which is called
the \emph{constant subring} of $(A, \sigma)$, denoted by $C_A$.
A difference ring $(A^*, \sigma^*)$ is called a \emph{difference extension} of $(A, \sigma)$ if $A\subseteq A^*$ and $\sigma^*\mid_A= \sigma$.
By abuse of notation, we will often write $\sigma$ for the extended automorphism $\sigma^*$ of $A^*$.

\begin{problem} [Indefinite Summation Problem]
  Let $(A^*, \sigma)$ be a specific difference extension of $(A, \sigma)$. Given $f\in A$, decide whether there exists $g\in A^*$
  such that $f = \sigma(g)- g$.  If such a $g$ exists, $f$ is said to be \emph{summable} in $A^*$.
\end{problem}

Abramov's algorithm~\cite{Abramov1971, Abramov1975, Abramov1995b} solves the indefinite summation problem
for rational functions. The indefinite hypergeometric summation problem was solved by Gosper's algorithm in~\cite{Gosper1978} and the more general P-recursive case without denominators was
solved by Abramov-van Hoeij's algorithm~\cite{AbramovHoeij1997,AbramovHoeij1999}.
As a discrete analogue of Risch's algorithm for elementary integration, Karr's algorithm~\cite{Karr1981, Karr1985} solves the indefinite summation problem
in a so-called $\Pi\Sigma$-extension of a given difference field. Karr's algorithm has been implemented and improved by Schneider in~\cite{schneider01, schneider08,Schneider2016}
with applications in physics~\cite{schneider16}.
The ideas of Karr have also been extended to higher order equations~\cite{HendriksSinger1999,bronstein00,schneider04c,schneider05c}.

The difference fields employed in Karr's algorithm are univariate rational function fields $K(t)$ with potentially complicated
ground fields~$K$. The summation problem is solved in such fields following Abramov's two-step approach: first find a $v\in K[t]$
such that every solution $g\in K(t)$ must have a denominator that divides $v$, and then find a $u\in K[t]$ such that $g=u/v$ is
a solution. In order to find such $u$ and~$v$, certain subproblems have to be solved for the ground field~$K$, and if $K$ is again
a univariate rational function field (perhaps still with a complicated ground field), then the procedure is applied recursively.
The recursion ends when the constant field is reached.

Parallel summation also follows Abramov's two-step approach, but instead of using univariate rational function fields with
potentially complicated ground fields, it allows the use of multivariate rational function fields, which may then have simpler
ground fields. The summand is thus given as an element of a difference field of the form $F=K(t_0,\dots,t_{n-1})$. In the first step,
we seek a $v\in K[t_0,\dots,t_{n-1}]$ such that every solution $g\in F$ must have a denominator that divides~$v$, and in the second
step, we then find a polynomial $u\in K[t_0,\dots,t_{n-1}]$ such that $g=u/v$ is a solution. 

Thus, while Karr's algorithm handles the generators of a difference field $K(t_0)(t_1)\cdots(t_{n-1})$ one after the other,
the parallel approach handles them ``in parallel''. How exactly this is done, this depends on the automorphism $\sigma$
of the difference field. There is little hope to obtain an algorithm that can execute both steps for arbitrary
difference fields of the form $K(t_0,\dots,t_{n-1})$. This is the same as in the integration case, where it can only be guaranteed
under strong restrictions on the differential field that the method does not overlook any solutions, and where
the method is still of interest as a valuable heuristic tool in situations where these strong restrictions do not apply.

In this paper, we restrict the attention to the ground field $K=C(x)$ where $C$ is a field of characteristic
zero. Together with the $C$-automorphism $\sigma\colon K\to\ K$ defined by $\sigma(x)=x+1$, we have that $(K, \sigma)$
is a difference field and its constant subfield is~$C$.
Let $R := K[t_0, \ldots, t_{n-1}]$ and $F := K(t_0, \ldots, t_{n-1})$.
The central problem of parallel summation is as follows.

\begin{problem}\label{Problem:parasum}
Given $f \in F$, decide whether there exists $g\in F$
  such that $f = \sigma(g)- g$.
\end{problem}

For a general $C$-automorphism $\sigma$ of~$F$,
the following example shows that the difference field $(F, \sigma)$ may contain new constants that are not in~$C$.

\begin{example}
Let $F = C(x, t_0, t_1)$ with the $C$-automorphism $\sigma$ satisfying $\sigma(x) = x+1 $, $\sigma(t_0) = t_1$, and $\sigma(t_1) = t_0 + t_1$.
Then $p = (t_1^2 - t_0^2 - t_0t_1)^2$ is a new constant in $F$.
\end{example}
In general, deciding the existence of new constants is a difficult problem. 
It is equivalent to finding algebraic relations among sequences. See~\cite{Kauers08}
for how to do this for C-finite sequences and Karr's algorithm~\cite{Karr81} for
how to solve it for $\Pi\Sigma$-fields.

The following example shows that in general, the denominator of $g$ may have
factors that are not related to any of the factors of the denominator of~$f$.

\begin{example}
\label{Exam:strange}
  Let $F = C(x, t_0, t_1)$ and $\si$ is the $C$-automorphism defined by $\si(x) = x+1$, $\si(t_0)=2t_0 + xt_1$ and $\si(t_1) = 2t_1$. Then
  \[\si\left(\frac{t_0}{t_1}\right) - \frac{t_0}{t_1}  = \frac{2t_0 + xt_1}{2t_1} - \frac{t_0}{t_1} =  \frac{x}{2}.\]
\end{example}

Similar phenomena can be expected if $\si(t_i)$ is not a polynomial for some~$i$. We are interested in predicting denominators of closed forms when $\si(t_i)$ is a polynomial for every~$i$. This motivates the following hypothesis.
\begin{hypo}\label{hypo}
The constant field of $(F, \sigma)$ is the field $C$ and $\sigma$ is also a $C$-automorphism of $R$.
\end{hypo}

To solve Problem~\ref{Problem:parasum}, one first needs to estimate the possible irreducible polynomials in the denominator of~$g$. To this end, we now extend the notion of special polynomials in parallel integration to the summation setting.
\begin{defi}\label{DEF:sepcial}
A polynomial $P \in R$ is said to be \emph{special} if there exist $i\in \bZ\setminus \{0\}$ such that
$P\mid \sigma^i(P)$ and it is said to be \emph{normal} if $\gcd(P, \sigma^i(P)) = 1$ for all $i\in \bZ\setminus \{0\}$.
A polynomial $P \in R$ is said to be \emph{factor-normal} if all of its irreducible factors are normal.
Two polynomials $P, Q\in R$ are said to be $\sigma$-equivalent if there exist $m\in \bZ$ and $u\in K$ such that
$P = u\cdot \sigma^m(Q)$.
\end{defi}
By the above definition, any nonzero element in $K$ is both special and normal and an irreducible polynomial in $R$
is either special or normal. The product of special polynomials is also special. If two normal polynomials are not
$\sigma$-equivalent, then their product is still normal.

Concerning special and normal polynomials, there are two basic and natural questions: firstly, how to decide if a given irreducible polynomial is special or normal?
Secondly, how to decide whether two polynomials are $\sigma$-equivalent or not?
We will answer these questions in next section for the difference field generated by P-recursive sequences.

\subsection{Local dispersions and denominator bounds}\label{SubSECT:disp}
Abramov in~\cite{Abramov1971} introduced the notion of dispersions for rational summation. It is a discrete analogue of the multiplicity.
We define a local version of Abramov's dispersions in $R$ at an irreducible normal polynomial, following~\cite{chen2021ISSACb}.
Let $p, Q\in R$ with $p$ being an irreducible normal polynomial. If $\sigma^i(p) \mid Q$ for some $i\in \bZ$, the \emph{local dispersion} of $Q$ at~$p$,
denoted by $\disp_p(Q)$, is defined as
the maximal integer distance $|i-j|$ with $i, j\in \bZ$ satisfying $\si^i(p)\mid Q$ and $\si^j(p)\mid Q$; otherwise we define $\disp_p(Q)=-\infty$.
Conventionally, we set $\disp_p(0) = +\infty$.
The (global) \emph{dispersion} of $Q$, denoted by $\disp(Q)$, is defined as
\[ \max \{\disp_p(Q)\mid \text{$p$ is an irreducible normal polynomial in $R$}\}.\]
Note that $\disp(Q) = -\infty$ if $Q\in R\setminus\{0\}$ has no irreducible normal factor.
 For a rational function $f = a/b\in F$ with $a, b\in R$, $\deg(b)\geq 1$ and $\gcd(a, b)=1$, we also define $\disp_p(f) = \disp_p(b)$ and $\disp(f) = \disp(b)$.
The set $\{\sigma^i(p)\mid i\in \bZ\}$ is called the $\sigma$-orbit at $p$, denoted by $[p]_{\sigma}$.
Note that $\disp_p(Q) = \disp_q(Q)$ if $q \in [p]_{\sigma}$. So we can define the local dispersion and dispersion of a rational function
at a $\sigma$-orbit.

The following lemma shows how the local dispersions and dispersions change under the action of the difference operator $\Delta$, which is
defined by $\Delta(f) = \si(f) -f$ for any $f\in F$.

\begin{lemma}\label{LEM:dis}
Let $f=a/b \in F$ with $a, b\in R$ and $\gcd(a, b)=1$ and let $p \in R$ be an irreducible normal factor of~$b$.
Then $\disp_p(\Delta(f)) = \disp_p(f) + 1$ and $\disp(\Delta(f)) = \disp(f) + 1$.
\end{lemma}
\begin{proof}
Let $d = \disp_p(b)$.  Without loss of generality, we may assume that $p\mid b$ but $\si^i(p)\nmid b$ for any $i<0$.
Since $\gcd(a, b)=1$ and $\si$ is a $C$-automorphism of $K[t_0, \ldots, t_{n-1}]$,  $\gcd(\si^i(a), \si^i(b))=1$ for any $i \in \bZ$.
We now write
\[ \sigma(f) - f = \frac{\si(a)b - a\si(b)}{b\si(b)} = \frac{A}{B}, \]
where $A, B \in K[t_0, \ldots, t_{n-1}]$ and $\gcd(A, B)=1$.
Since $p\mid b$ but $p\nmid a\si(b)$, we have $p \nmid (\si(a)b - a\si(b))$ and then $p \nmid A$.
By the definition of local dispersions, $\si^d(p)\mid b$ but $\si^{d+1}(p)\nmid b$. Since $\gcd(a, b)=1$, we have $\si^{d}(p)\nmid a$ and then $\si^{d+1}(p)\nmid \si(a)$.
Then $\si^{d+1}(p) \nmid \si(a)b$, which implies $\si^{d+1}(p) \nmid (\si(a)b - a\si(b))$ and also $\si^{d+1}(p) \nmid A$.
So $p\mid B$ and $\si^{d + 1}(p)\mid B$, which implies that $\disp_p(B)\geq d+1$. Since $B\mid b\si(b)$, we have
 $\disp_p(B)\leq\disp_p(b\si(b))= d+1$. Therefore, $\disp_p(B) = d+1$. Since the equality $\disp_p(B) = d+1$ holds for all irreducible normal factors,
 we have $\disp(\Delta(f)) = \disp(f) + 1$.
\end{proof}
By the above lemma, we get that $f$ is not $\si$-summable in $F$ if $\disp(f) = 0$. 
If we know how to detect the $\si$-equivalence in $R$, then we can
write a given polynomial $P\in R$ as $ P = P_s \cdot P_n$, all irreducible factors of $P_s \in R$ are special
and all irreducible factors of $P_n\in R$ are normal. We call $(P_s, P_n)$ the \emph{splitting factorization} of $P$ and  $ P_n$
the \emph{normal part} of $P$.

 \begin{thm}\label{THM:normaldenom}
 Let $f\in F$ and $v_n \in R$ be the normal part of the denominator of~$f$.
 If $f = \si(g) - g$ for some $g\in F$, then the normal part of the denominator of $g$ divides the polynomial
 \[\gcd\left(\prod_{i=0}^{d}\sigma^{i}(v_n),\prod_{i=0}^{d}\sigma^{-i-1}(v_n )\right),\]
 where $d := \disp(g) = \disp(f)-1$.
 \end{thm}

\begin{proof}
Write $f= u/v \in F$ with $u,v \in R$ and $gcd(u,v)=1$. Assume that the splitting factorization of $v$ is $(v_s, v_n)\in R^2$.
If $f = \si(g)-g$ for some $g\in F$, we also write $g = p/q$ with $p,q\in R$ and $gcd(p,q)=1$ and let $(q_s, q_n)$ be the splitting factorization of~$q$.
By Lemma~\ref{LEM:dis}, we have $d:= \disp(q_n) = \disp(v_n) -1$.  We now show
\begin{equation}\label{EQ:div}
q_n \mid\gcd\left(\prod_{i=0}^{d}\sigma^{i}(v_n),\prod_{i=0}^{d}\sigma^{-i-1}(v_n )\right).
\end{equation}
We first show that $q_n\mid \prod_{i=0}^{d}\sigma^{i}(v_n)$.
The equality $f = \sigma(g)-g$ implies that
  \begin{equation}\label{eq:resf}
  g=\frac{v\sigma(g)-u}{v}
  \end{equation}
 Applying $\sigma$ to both sides of the above equation yields
 \[\sigma(g)=\frac{\sigma(v)\sigma^2(g)-\sigma(u)}{\sigma(v)}.\]
Substituting $\sigma(g)$ in the equation~\eqref{eq:resf} yields
 \[ g =\frac{1}{v}\left(v\cdot\frac{\sigma(v)\sigma^2(g)-\sigma(u)}{\sigma(v)}-u\right)\]
After $d$ repetitions of the above process, we get
  \[ g =\frac{a\cdot \sigma^{d+1}(g)-b}{v\sigma(v) \cdots \sigma^d(v)}\]
for some $a,b \in R$. The denominator of the $g$ is $q=q_sq_n$, while the denominator of the right-hand side of the above equality
is a divisor of $V := v\sigma(v)\cdots \sigma^{d}(v)\sigma^{d+1}(q)$. Then $q_n \mid V$.
Let $(V_s, V_n)$ be the splitting factorization of $V$. Then $V_n =  v_n\sigma(v_n)\cdots \sigma^{d}(v_n) \sigma^{d+1}(q_n)$ and $q_n \mid V_n$.
Since $\disp(q_n)=d$, we have $\gcd(q_n, \sigma^{d+1}(q_n)) = 1$. Hence we have $q_n \mid v_n\sigma(v_n)\cdots \sigma^{d}(v_n)$.
The proof of the divisibility $q_n\mid \prod_{i=0}^{d}\sigma^{-i-1}(v_n)$ is analogous. So the divisibility~\eqref{EQ:div} holds.
\end{proof}


\begin{example}
Let $F=C(x)(t_0, t_1)$ with a $C$-automorphism defined by $\si(x) = x+1, \si(t_0) = t_1$ and $\si(t_1) = -6t_0 + 5t_1$.
Consider the equation
\[f = \frac{636t_0^3+443t_0^2t_1-1428t_0t_1^2+565t_1^3}{2592(3t_0-2t_1)^2(t_0-t_1)^2(2t_0-t_1)(t_0+t_1)} = \si(y) - y.\]
We now decide whether this equation has a solution in~$F$.
Firstly, we can detect that the irreducible factor $2t_0-t_1$ is special and other irreducible factors are normal.
Then the normal part of the denominator of $f$ is $B := (3t_0-2t_1)^2(t_0-t_1)^2(t_0+t_1)$ with dispersion $d = 2$. By Theorem~\ref{THM:normaldenom}, the normal part of
the denominator of any solution $g$ divides the polynomial $36(t_1+t_0)^3(t_1-t_0)^2$. Then we can make an ansatz for $g$ as
\[g = \frac{U}{36(t_1+t_0)^3(t_1-t_0)^2(2t_0-t_1)},\]
where $U\in C(x)[t_0, t_1]$ satisfying the recurrence equation
\[
(t_1+t_0)^3\si(U)-72(t_1-t_0)(2t_1-3t_0)^2 U = b,
\]
where $b = (t_1+t_0)^26(t_1-t_0)(636t_0^3+443t_0^2t_1-1428t_0t_1^2+565t_1^3)$. We can bound the degree of\, $U$ which is~$3$. Then we get $U=t_0^3+4t_0^2+5t_0t_1^2+2t_1^3$ by solving a linear difference system for rational solutions.
So we have the rational solution
\[
g =\frac{2t_1+t_0}{36(t_1-2t_0)(t_0+t_1)(t_1-t_0)^2}.
\]
\end{example}

We will discuss how to estimate special factors in the denominator of $g$ in next section for the difference
fields generated by P-recursive sequences.

\subsection{Number of irreducible special polynomials}\label{sec:numberirredspecial}
In Section~\ref{SubSECT:disp}, we have already outlined the procedure for computing the normal part of the denominator of $g$ satisfying $\sigma(g)-g=f$. The challenge that remains is to handle the special part. As demonstrated in Example~\ref{Exam:strange}, a peculiar situation arises where the denominator of $g$ contains a special polynomial that does not already appear in the denominator of~$f$. Hence, to determine the denominator of~$g$, it becomes necessary to identify all irreducible special polynomials. However, the computation of all special polynomials remains an unresolved issue at present. In this subsection, we aim to establish that the number of irreducible special polynomials in $R$ that do not pairwise differ by elements of $K^*$ is bounded by the number of generators of $R$ over $C(x)$. This provides a crucial insight into the limited diversity of irreducible special polynomials. In Section~\ref{SUBSECT:structureofspecialpolynomials}, under certain assumptions, we will unveil the structure of irreducible special polynomials.
We begin with the following lemma that is a direct consequence of Theorem 2.1.12 on page 114 of \cite{Levin2008}.
\begin{lemma}
	\label{lm:constants}
	Suppose that $\cF$ is a $\sigma$-field of characteristic zero with algebraically closed field $C$ of constants, and $f\in \cF$ satisfying that $\sigma^\ell(f)=f$ for some $\ell>0$. Then $f\in C$.
\end{lemma}
\begin{corollary}
	\label{cor:sumofspecialpolynomials}
	Suppose that $p_1, p_2,\dots, p_m$ are special polynomials that are linearly independent over $K$, $\alpha_2,\dots,\alpha_m\in K$. Then $p_1+\alpha_2 p_2+\dots+\alpha_m p_m$ is a special polynomial if and only if $\alpha_2=\cdots=\alpha_m=0$.
\end{corollary}
\begin{proof}
	It suffices to show the necessary part. Suppose that $p_1+\alpha_2 p_2+\dots+\alpha_m p_m$ is a special polynomial. Then there is a positive integer $\ell$ and $\gamma,\beta_1,\dots,\beta_m\in K$ such that $\sigma^\ell(p_i)=\beta_i p_i$ and
	$$
	\sigma^\ell(p_1+\alpha_2 p_2+\dots+\alpha_m p_m)=\gamma (p_1+\alpha_2 p_2+\dots+\alpha_m p_m).
	$$
	A straightforward calculation reveals that $\beta_1=\gamma$, and $\sigma^\ell(\alpha_i)\beta_i=\beta_1\alpha_i$ for all $2\leq i\leq m$. This implies that $\sigma^\ell(\alpha_i p_i)=\beta_1 \alpha_i p_i$ and consequently, $\sigma^{\ell}(\alpha_i p_i/p_1)=\alpha_i p_i/p_1$, for all $2\leq i \leq m$. According to Lemma~\ref{lm:constants}, $\alpha_i p_i/p_1\in C$. For each $2\leq i\leq m$, as $p_i$ and $p_1$ are linearly independent over~$K$, it follows that $\alpha_i=0$.
\end{proof}
\begin{proposition}
	\label{prop:algebraicindependence}
	Suppose that $p_1,\dots, p_m$ are irreducible special polynomials that are pairwise not shift equivalent. Denote by $\ell_i$ the smallest positive integer such that $p_i \mid \sigma^{\ell_i}(p_i)$. Then the $\sigma^j(p_i), i=1,\dots,m, j=0,\dots,\ell_i-1$ are algebraically independent over~$K$.
\end{proposition}
\begin{proof}
	Set $N =\lcm(\ell_1,\dots,\ell_m)$. Then $\sigma^N(\sigma^j(p_i))=\alpha_{i,j}\sigma^j(p_i)$ for some $\alpha_{i,j}\in K$.
	Suppose on the contrary that the $\sigma^j(p_i), i=1,\dots,m, j=0,\dots,\ell_i-1$ are algebraically dependent over $K$.  Due to the difference analogue of Kolchin-Ostrowski theorem (see \cite{Hardouin2008-hypertranscendance,Ogawara2017}), there are integers $d_{i,j}, i=1,\dots,m,j=0,\dots,\ell_i-1$, not all zero, and $\beta\in K^*$ such that
	$$
	\prod_{i=1}^m \prod_{j=0}^{\ell_i-1} \alpha_{i,j}^{d_{i,j}}=\frac{\sigma^N(\beta)}{\beta}.
	$$
	Since $\sigma^N( \prod_{i,j} \sigma^j(p_i)^{d_{i,j}})=  \prod_{i,j} \alpha_{i,j}^{d_{i,j}} \prod_{i,j} \sigma^j(p_i)^{d_{i,j}}$, we have that
	$$
	\sigma^{N}\left(\frac{ \prod_{i,j} \sigma^j(p_i)^{d_{i,j}}}{\beta}\right)= \frac{\prod_{i,j} \sigma^j(p_i)^{d_{i,j}}}{\beta}.
	$$
	Due to Lemma~\ref{lm:constants}, $ \prod_{i,j} \sigma^j(p_i)^{d_{i,j}}=c\beta$ for some $c\in C$.
	Denote $S_1=\{(i,j) \mid d_{i,j}>0\}$ and $S_2=\{(i,j) \mid d_{i,j}<0\}$. Since the $d_{i,j}$ are not all zero and $t_0,\dots,t_{n-1}$ are algebraically independent over~$K$, neither $S_1$ nor $S_2$ is empty. This leads to
	$$
	\prod_{(i,j)\in S_1}\sigma^j(p_i)^{d_{i,j}}=c\beta\prod_{(i,j)\in S_2}\sigma^j(p_i)^{-d_{i,j}}.
	$$
	Choose $(i_1,j_1)\in S_1$.  Then $\sigma^{j_1}(p_{i_1})$ divides $\prod_{(i,j)\in S_2}\sigma^j(p_i)^{-d_{i,j}}$ and thus there exists $(i_2,j_2)\in S_2$ such that $\sigma^{j_1}(p_{i_1})$ divides $\sigma^{j_2}(p_{i_2})$.  As both  $\sigma^{j_1}(p_{i_1})$ and $\sigma^{j_2}(p_{i_2})$ are irreducible, $\sigma^{j_1}(p_{i_1})=\gamma\sigma^{j_2}(p_{i_2})$ for some $\gamma\in K$. If $i_1=i_2$ then $0\leq j_1\neq j_2\leq \ell_{i_1}-1$. Without loss of generality, assume $j_2>j_1$. Then $\sigma^{j_2-j_1}(p_{i_1})=p_{i_1}/\sigma^{-j_1}(\gamma)$, which contradicts the minimality of $\ell_{i_1}$. If $i_1\neq i_2$ then $p_{i_1}$ and $p_{i_2}$ are shift equivalent, contradicting the initial assumption.
\end{proof}

Since $F$ is a field generated over $K$ by $n$ indeterminates, the transcendence degree of $F$ over $K$ is equal to $n$. So we have the following corollary.
\begin{corollary}
	\label{cor:numbersofspecialpolynomials}
	Suppose that $p_1,\dots, p_m$ are irreducible special polynomials that are pairwise not shift equivalent. Then $\sum_{i=1}^m \ell_i \leq n$, where $\ell_i$ is the smallest positive integer such that $p_i \mid \sigma^{\ell_i}(p_i)$.
\end{corollary}

\section{The P-recursive case}\label{sec:p-finite}

P-recursive sequences, introduced by Stanley~\cite{stanley1980}, satisfy linear recurrence equations with polynomial coefficients.
The generating function of a P-recursive sequence is a D-finite function, which satisfies a linear differential equations with polynomial coefficients.
This class of sequences has been extensively studied in combinatorics~\cite{zeilberger90, Stanley1999} and symbolic computation~\cite{kauers10j, Kauers2023} together with its generating functions.
In this section, we will focus on parallel summation in difference fields generated by P-recursive sequences.

Let $F$ be the field $C(x)(t_0, \ldots, t_{n-1})$ with a $C$-automorphism $\si$ satisfying that
$\si(x) = x+1$, $\si(t_0) = t_1, \ldots, \si(t_{n-2}) = t_{n-1}$, and
\[\si(t_{n-1}) = a_0 t_0 + \cdots + a_{n-1}t_{n-1},\]
where $a_0, \ldots, a_{n-1} \in C(x)$ and $a_0\neq 0$.  So $\si$ is a $C$-automorphism of the ring $R= C(x)[t_0, \ldots, t_{n-1}]$.
We still assume in this section that the constant field of $(F, \si)$ is the field~$C$, even though this condiction is not easy to check for a given field.
In order to study the indefinite summation problem in~$F$, we first address two basic questions
on special and normal polynomials. In Section~\ref{SUBSECT:structureofspecialpolynomials}, we prove some structural properties on special polynomials under certain assumptions. In Section~\ref{SUBSECT:equivalence}, we
will answer the question of deciding whether two irreducible polynomials in $R$ are $\si$-equivalent or not.

\subsection{Degrees of irreducible special polynomials}\label{SUBSECT:structureofspecialpolynomials}
In the P-recursive case, by Lemma~\ref{lm:homogeneous} below, computing all special polynomials of degree $m$ is equivalent to computing all hypergeometric solutions of the $m$th symmetric power of the system
\begin{equation}
	\label{eqn:differencesystem}
	\sigma(Y)=AY
\end{equation}
or $\sigma^s(Y)=A_{(s)}Y$ for some $s>1$, where
\[
  A=\begin{pmatrix} 0 & 1 &  & &  \\  & 0 & 1 &  & \\ && \ddots & \ddots &  \\ &  & & 0&1 \\ a_0 & a_1 & a_2 & \dots & a_{n-1} \end{pmatrix}
  \]
 and $A_{(s)}=\sigma^{s-1}(A)\dots \sigma(A)A$.
	Algorithms for computing all hypergeometric solutions of a given linear difference equation are known, for example, refer to~\cite{Petkovsek1998}. Corollary~\ref{cor:numbersofspecialpolynomials} establishes the existence of a degree bound for all irreducible special polynomials.  However, by the absence of a known degree bound, the computation of all irreducible special polynomials remains an unresolved challenge. In this subsection, we will prove that when $\sum_{i=1}^m \ell_i=n$ with $\ell_i$ as defined in Corollary~\ref{cor:numbersofspecialpolynomials}, all irreducible special polynomials are linear in $t_0,t_1,\dots,t_{n-1}$. Consequently, in this specific case, the degree of all irreducible special polynomials is exactly equal to 1 and thus we can compute all irreducible special polynomials.
\begin{lemma}
\label{lm:homogeneous}
	All special polynomials are homogeneous.
\end{lemma}
\begin{proof}
	Suppose that $p$ is a special polynomial and is not homogeneous. Write $p=\sum_{i=0}^m p_i$ where $p_i$ is the $i$th homogeneous part of $p$ and $p_m\neq 0$. Assume that $\sigma^\ell(p)=\alpha p$ for some nonzero $\alpha\in C(x)$. Then
	$\sigma^\ell(p)=\sum_{i=0}^m \sigma^\ell(p_i)=\alpha p=\sum_{i=0}^m \alpha p_i$. Note that $\sigma^\ell(p_i)$ is also homogeneous of degree~$i$. We have that $\sigma^\ell(p_i)=\alpha p_i$ for all $0\leq i \leq m$. Since $p$ is not homogeneous, there is an $i_0$ such that $p_{i_0}\neq 0$. Hence $\sigma^\ell(p_{i_0}/p_m)=p_{i_0}/p_m$. According to Lemma~\ref{lm:constants}, $p_{i_0}/p_m\in C$, which contradicts the fact that the numerator and denominator of $p_{i_0}/p_m$ have different degrees.
\end{proof}

We start with the $C$-finite case. In this case, we will demonstrate that the degree of all irreducible special polynomials is always equal to~$1$, without requiring the assumption that $\sum_{i=1}^m \ell_i = n$.
\begin{proposition}
	\label{prop:cfinitecase} Suppose that $A\in \GL_n(C)$. Then all irreducible special polynomials are linear in $t_0,t_1,\dots,t_{n-1}$.
\end{proposition}
\begin{proof}
	Let $B\in \GL_n(C)$ such that $BAB^{-1}=\diag(J_1,J_2,\dots,J_\ell)$, where $J_i$ is a Jordan block of order $n_i$. We claim that $n_i=1$ for all $1\leq i \leq \ell$. Without loss of generality, assume that $n_1>1$ and $\alpha_1$ is the eigenvalue of $J_1$.
	Set $\bar{T}=(\bar{t}_0,\dots,\bar{t}_{n-1})^t=B(t_0,\dots,t_{n-1})^t$. Then $\sigma(\bar{T})=BAB^{-1}\bar{T}$. Therefore $\sigma(\bar{t}_{n_1-1})=\alpha_1 \bar{t}_{n_1-1}+\bar{t}_{n_1}$ and $\sigma(\bar{t}_{n_1})=\alpha_1\bar{t}_{n_1}$. From these, it follows that $\sigma(\frac{\bar{t}_{n_1-1}}{\bar{t}_{n_1}})=\frac{\bar{t}_{n_1-1}}{\bar{t}_{n_1}}+\frac{1}{\alpha_1}$. Consequently,
	$$
	\sigma\left(\frac{\bar{t}_{n_1-1}}{\bar{t}_{n_1}}-\frac{x}{\alpha_1}\right)=\frac{\bar{t}_{n_1-1}}{\bar{t}_{n_1}}-\frac{x}{\alpha_1}.
	$$
	In other words, $\bar{t}_{n_1-1}/\bar{t}_{n_1}-x/\alpha_1\in C$, which is a contradiction with Hypothesis~\ref{hypo}. This proves our claim. Therefore $BAB^{-1}=\diag(\alpha_1,\dots,\alpha_n)$, where $\alpha_i\in C$. Finally, suppose that $p$ is an irreducible special polynomial. Note that $p$ can be expressed as a polynomial in $\bar{t}_0,\bar{t}_1,\dots,\bar{t}_{n-1}$. Corollary~\ref{cor:sumofspecialpolynomials} implies that $p$ is a monomial in $\bar{t}_0,\bar{t}_1,\dots,\bar{t}_{n-1}$. Hence $p=\beta \bar{t}_i$ for some $0\leq i \leq n-1$ and $\beta\in K$, and so it is linear in $t_0,t_1,\dots,t_{n-1}$.
\end{proof}
\begin{remark}
	\label{rem:cfinite}
	In the proof of Proposition~\ref{prop:cfinitecase}, the special polynomials $\bar{t}_0,\dots,\bar{t}_{n-1}$ are pairwise not shift equivalent and then the condition $\sum_{i=1}^m \ell_i=n$ is automatically satisfied. In fact, suppose $\sigma(\bar{t}_{i_1})=\beta\bar{t}_{i_2}$ for some $0\leq i_1\neq i_2\leq n-1$ and $\beta\in K$. Since $\sigma(\bar{t}_{i_1})=\alpha\bar{t}_{i_1}$ for some $\alpha\in K$, it follows that $\bar{t}_1,\bar{t}_2$ are linearly dependent over~$K$, which contradicts the fact that $\bar{t}_0,\dots,\bar{t}_{n-1}$ are algebraically independent over~$K$.
\end{remark}
Before proceeding to the general case, let's recall some fundamental results from difference Galois theory. For detailed information, readers can refer to Chapter~1 of \cite{vanderPutSinger1997-difference}. Let $\cR$ be the Picard-Vessiot ring for $\sigma(Y)=AY$ over~$K$, where $A$ is given as in \eqref{eqn:differencesystem}. In $\cR$, there exist idempotents $e_0,e_1,\dots,e_{s-1}$ such that
$$
\cR=\cR_0\oplus \cR_1 \oplus \dots \oplus \cR_{s-1},
$$
where $\cR_i=e_i\cR$ and $\cR_i$ is a domain. Moreover, $\cR_i$ serves as the Picard-Vessiot ring for $\sigma^s(Y)=A_{(s)}Y$ over $K$ with
$$
A_{(s)}=\sigma^{s-1}(A)\dots\sigma(A)A.
$$
Let $G$ be the Galois group of $\sigma(Y)=AY$ over $K$ and $H$ be the Galois group of $\sigma^s(Y)=A_{(s)}Y$ over~$K$. By Corollary~1.17 \cite[p.13]{vanderPutSinger1997-difference} we have $[G:H]=s$ and consequently, $H$~contains $G^{\circ}$, the identity component of~$G$. On the other hand, due to Proposition~1.20, $\cR_i$ is a trivial $H$-torsor which implies that $H$ is connected since $\cR_i$ is a domain. Hence $H=G^\circ$. In the following, it will be shown that the hypothesis $\sum_{i=1}^m \ell_i=n$ implies that the group $G^\circ$ is a torus. Consequently, the system $\sigma^s(Y)=A_{(s)}Y$ is equivalent to a diagonal system, indicating that appropriate linear combinations of $t_0,t_1,\dots,t_{n-1}$ are special polynomials. Indeed, they are all irreducible special polynomials.\begin{lemma}
	\label{lm:goodfundamentalmatrix}
	Suppose that $p_1,\dots, p_m$ are special polynomials. Then there exists a fundamental matrix $\cZ\in \GL_n(\cR)$ of $\sigma(Y)=AY$ such that $p_i(\cZ_j)$ is invertible in $\cR$ for all $1\leq i \leq m, 1\leq j \leq n$, where $\cZ_j$ denotes the $j$th column of $\cZ$.
\end{lemma}
\begin{proof}
	Let $N$ be a positive integer such that $p_i\mid \sigma^N(p_i)$ for all $1\leq i\leq m$ and let $q=(\prod_{i=1}^mp_i) \sigma(\prod_{i=1}^m p_i)\cdots \sigma^{N-1}(\prod_{i=1}^m p_i)$. Then $\sigma(q)=\alpha q$ for some $\alpha\in K$. We first show the lemma for $q$.
	
	Let $\cZ\in \GL_n(\cR)$ be a fundamental matrix of $\sigma(Y)=AY$. We claim that there exists an $M\in \GL_n(C)$ such that $q((\cZ M)_j)\neq 0$. Let $\bfu=(u_1,\dots,u_n)^t$ be a vector with indeterminate entries. Since $\cZ$ is invertible, as a polynomial in $\cR[\bfu]$, $q(\cZ \bfu)\neq 0$. Write $q(\cZ \bfu)=\sum_{j=1}^d f_{j}(\bfu) \bfm_j$, where $f_{j}(\bfu)\in C[\bfu]$ and $\bfm_1,\dots,\bfm_d\in \cR$ are linearly independent over~$C$. As $q(\cZ\bfu)\neq 0$, at least one of $f_1(\bfu),\dots,f_d(\bfu)$ is not zero, say $ f_{j_1}(\bfu)\neq 0$. Set $U$ to be the Zariski open subset of $C^n$ consisting of all $\bfa$ in $C^n$ such that $f_{j_1}(\bfa)\neq 0$. Then $U\times \dots \times U$ is a non-empty Zariski open subset of $C^{n\times n}$, where the direct product takes $n$ times. Let $M\in U\times \dots \times U$ be such that $\det(M)\neq 0$. Such $M$ exists because $U\times \dots \times U$ is Zariski dense in $C^{n\times n}$. Then for each column $\bfc$ of~$M$, it follows that $f_{j_1}(\bfc)\neq 0$, and thus $q(\cZ \bfc)\neq 0$. Since $M$ is invertible, $\cZ M$ is also a fundamental matrix of $\sigma(Y)=AY$. This proves our claim.
	
	Let $\bfc$ be a column of $M$. Then
	$$
	\sigma(q(\cZ \bfc))=q^\sigma(A\cZ\bfc)=\sigma(q)(\cZ\bfc)=\alpha q(\cZ\bfc)
	$$
	where $q^\sigma$ denotes the polynomial obtained by applying $\sigma$ to the coefficients of $q$. Hence, $q(\cZ\bfc)$ generates a nonzero $\sigma$-ideal in $\cR$. Since $\cR$ is $\sigma$-simple, this ideal must be equal to $\cR$ and thus $q(\cZ\bfc)$ is invertible in $\cR$. Finally, for each $1\leq i\leq m$, since $q(\cZ\bfc)=p_i(\cZ\bfc)h_i$ for some $h_i\in \cR$, $p_i(\cZ\bfc)$ is invertible in $\cR$.
\end{proof}

\begin{lemma}
	\label{lm:dimensionofGaloisgroups}
	Suppose that there exist irreducible special polynomials $p_1,\dots, p_m$ that are pairwise not shift equivalent, satisfying that $\sum_{i=1}^m \ell_i=n$, where $\ell_i$ is the smallest positive integer such that $p_i\mid \sigma^{\ell_i}(p_i)$. Then $\dim(G)=n$.
\end{lemma}
\begin{proof}
	It suffices to show that $\dim(G^\circ)=n$.
	Set
	$$
	  q_{\ell_0+\ell_1+\dots+\ell_i+j}=\sigma^{j-1}(p_{i+1})
	$$
	where $0\leq i\leq m-1$, $1\leq j \leq \ell_i$, and $\ell_0=0$. By Lemma~\ref{lm:goodfundamentalmatrix}, there exists a fundamental matrix $\cZ\in\GL_n(\cR) $ of $\sigma(Y)=AY$ such that $q_i(\cZ_j)$ is invertible in $\cR$ for all $1\leq i,j\leq n$, where $\cZ_j$ denotes the $j$th column of $\cZ$.
	Consider $\cR_0$, the Picard-Vessiot ring for $\sigma^s(Y)=A_{(s)}Y$ over~$K$. Let $\cF$ be the field of fractions of $\cR_0$. Since $G^\circ$ is the Galois group of $\sigma^s(Y)=A_{(s)}Y$ over $K$, $\dim(G^\circ)=\trdeg(\cF/K)$. Note that $e_0\cZ$ is a fundamental matrix of $\sigma^s(Y)=A_{(s)}Y$ and $e_0q_i(\cZ_j)=q_i(e_0\cZ_j)$ is invertible in $\cR_0$. Set $N={\rm lcm}(\ell_1,\dots,\ell_m)$. Then $q_i\mid \sigma^{sN}(q_i)$ for all $1\leq i\leq n$. Suppose that $\sigma^{sN}(q_i)=\alpha_i q_i$ with $\alpha_i\in K$. Then for each $1\leq i\leq n$, it holds that $\sigma^{sN}(q_i(e_0\cZ_j))=\alpha_i q_i(e_0\cZ_j)$ for all $1\leq j \leq n$.
	
	We claim that for each $1\leq j \leq n$,  $q_1(e_0\cZ_j),\dots,q_n(e_0\cZ_j)$ are algebraically independent over~$K$. Assume on the contrary that $q_1(e_0\cZ_j),\dots,q_n(e_0\cZ_j)$ are algebraically dependent over $K$. Using an argument similar to that in the proof of Proposition~\ref{prop:algebraicindependence}, we obtain
	integers $d_i$, not all zero, and a nonzero $\beta\in K$ such that
	$$
	\sigma^{sN}\left(\frac{\prod_iq_i^{d_i}}{\beta}\right)=\frac{\prod_iq_i^{d_i}}{\beta}.
	$$
	Due to Lemma~\ref{lm:constants}, $\prod_iq_i^{d_i}=c\beta$ for some $c\in C$.
	In other words, $q_1,\dots,q_n$ are algebraically dependent over~$K$. This contradicts the conclusion of Proposition~\ref{prop:algebraicindependence}. The claim is established.
	
	Now, for $1\leq j_1\neq j_2\leq n$, we have that
	$$\sigma^{sN}(q_i(e_0\cZ_{j_1})/q_i(e_0\cZ_{j_2}))=q_i(e_0\cZ_{j_1})/q_i(e_0\cZ_{j_2}).$$
	By Lemma~\ref{lm:constants} (replacing $\sigma$ with $\sigma^s$), $q_i(e_0\cZ_{j_1})=c_{i,j_1,j_2}q_i(e_0\cZ_{j_2})$ for all $1\leq i\leq n$, where $c_{i,j_1,j_2}\in C$. Denote by $\tilde{\cF}$ the subfield of $\cF$ generated by all $q_i(e_0\cZ_j)$ over $K$. Then the previous discussion implies that
	$
	\trdeg(\tilde{\cF}/K)=n.
	$
	Note that for each $1\leq j \leq n$, every entry of $e_0\cZ_j$ is algebraic over
        $K(q_1(e_0\cZ_j),\dots, q_n(e_0\cZ_j))$
        (and thus algebraic over $\tilde{\cF}$),
  because $q_1(e_0\cZ_j),\dots,q_n(e_0\cZ_j)$ are algebraically independent over $K$ and they are polynomial in the entries of $e_0\cZ_j$. Hence $\cF$ is a finite algebraic extension of $\tilde{\cF}$, as $\cF=K(e_0\cZ)$. So $\trdeg(\cF/K)=n$ and then $\dim(G^\circ)=n$. Consequently, $\dim(G)=n$.
\end{proof}

\begin{theorem}
	\label{thm:systemwithmanyspecialpolynomials}
	Under the same assumption as in Lemma~\ref{lm:dimensionofGaloisgroups}, all irreducible special polynomials are linear in $t_0, t_1,\dots,t_{n-1}$.
\end{theorem}
\begin{proof}
	Let $q_i,\alpha_i$ be as in the proof of Lemma~\ref{lm:dimensionofGaloisgroups}.
	We first show that $G^\circ$ is a torus. Due to Lemma~\ref{lm:dimensionofGaloisgroups}, $\dim(G^\circ)=n$. Hence the rank of $X(G^\circ)$, the group of characters of $G^\circ$ (which is a free abelian group), is at most $n$. As $\sigma^{sN}(q_i(e_0\cZ_1))=\alpha_i q_i(e_0\cZ_1)$, for each $g\in G^\circ$, $\sigma^{sN}(g(q_i(e_0\cZ_1)))=\alpha_i g(q_i(e_0\cZ_1))$. Hence $g(q_i(e_0\cZ_1))=\chi_i(g)g(q_i(e_0\cZ_1))$, where $\chi_i(g)\in C$. In other words, $q_i(e_0\cZ_1)$ induces a character $\chi_i\in X(G^\circ)$. Suppose that there are integers $d_i$, not all zero, such that $\prod_i\chi_i^{d_i}=\id$, where $\id$ is the unitary of $X(G^\circ)$. Then for all $g\in G^\circ$,
	\begin{align*}
		g\left(\prod_iq_i(e_0\cZ_1)^{d_i}\right)&=\prod_i\chi_i^{d_i}(g)\prod_iq_i(e_0\cZ_1)^{d_i}=\prod_iq_i(e_0\cZ_1)^{d_i}.
	\end{align*}
	The Galois correspondence (see, for example, Lemma~1.28 on page~20 of \cite{vanderPutSinger1997-difference}) implies that $\prod_iq_i(e_0\cZ_1)^{d_i}\in K$, which contradicts the fact that $q_1(e_0\cZ_1),\dots,q_n(e_0\cZ_1)$ are algebraically independent over~$K$. Therefore the rank of $X(G^\circ)$ is exactly equal to~$n$. Let $\tilde{\chi}_1, \dots,\tilde{\chi}_n$ be a base of $X(G^\circ)$, as a free abelian group. Consider the morphism
	\begin{align*}
		\phi\colon G^\circ &\longrightarrow \GL_1(C)^n \\
		g&\longmapsto (\tilde{\chi}_1(g),\dots,\tilde{\chi}_n(g)).
	\end{align*}
	Then $\phi$ is surjective with finite kernel because $\dim(G^\circ)=n$. Moreover, due to Lemma~B.20 of \cite{Feng-Hrushovskialgorithm}, $\ker(\phi)$ is generated as an algebraic group by all unipotent elements of $G^\circ$.  Note that if $h\in \GL_n(C)$ is unipotent then $h$ is of finite order if and only if $h=I$, the identity matrix.  So $\ker(\phi)=\{I\}$ and $\phi$ is an isomorphism. This proves that $G^\circ$ is a torus.
	
	By Theorem~2.1 of \cite{HendriksSinger1999}, there exists a $T\in \GL_n(K)$ such that
	$$
	\sigma^s(T)A_{(s)}T^{-1}=\diag(b_1,\dots,b_n),
	$$
	where $b_i\in K$.  Set $(\bar{t}_0,\dots,\bar{t}_{n-1})^t=T(t_0,\dots,t_{n-1})^t$. Then $\sigma^s(\bar{t}_i)=b_i\bar{t}_i$ for all $0\leq i \leq n-1$. In other words, $\bar{t}_0,\dots,\bar{t}_{n-1}$ are special polynomials. Finally, using an argument similar to the proof of Proposition~\ref{prop:cfinitecase}, it follows that every irreducible special polynomial is linear in $t_0,t_1,\dots,t_{n-1}$.
\end{proof}

%
%
%

\subsection{The $\si$-Equivalence Problem} \label{SUBSECT:equivalence}
We now present a method for deciding whether two irreducible polynomials $p, q\in R =K[t_0,\dots,t_{r-1}]$ are
$\si$-equivalent or not.

If one of $p$ and $q$ is special, say $p$, then there exists a minimal positive integer $m$ (not greater than $n$ by Corollary~\ref{cor:numbersofspecialpolynomials}) such that $p\mid \si^m(p)$.
To decide whether $p$ and $q$ are $\si$-equivalent, it suffices to check whether $q$ is associate over $K$ to one of elements
in the set $\{p, \si(p), \ldots, \si^{m-1}(p)\}$. It remains to consider the case in which both  $p$ and $q$ are normal.

We now assume that $p,q\in R$ are irreducible and normal in~$R$.
We want to decide whether there exist $i\in\set Z$ and $u\in K\setminus\{0\}$ such that $\sigma^i(p)=uq$.
Observe first that there can be at most one such~$i$. For, if $i,i'$ and $u,u'$ are such that $\sigma^i(p)=uq$ and $\sigma^{i'}(p)=u'q$,
then $\sigma^{i}(p)/\sigma^{i'}(p)=u/u'$, so $p\mid\sigma^{i'-i}(p)$, and since $p$ is not special, we must have $i=i'$.
Observe also that for a given candidate $i\in\set Z$, it is easy to check whether there exists a $u$ with $\sigma^i(p)=uq$.
Therefore, it suffices to determine a finite list of candidates for~$i$.

Let $L,M\in K[S]$ be the (unique) monic minimal order annihilating operators of $p$ and~$q$, respectively. Let $s$ be their
order. Note that $p$ and $q$ cannot be shift equivalent if the orders of $L$ and $M$ are distinct.
Write $L=S^s+\ell_{s-1}S^{s-1}+\cdots+\ell_0$ and $M=S^s+m_{s-1}S^{s-1}+\cdots+m_0$.
By the minimality of the order of $L$ and $M$, we have $\ell_0,m_0\neq0$.

For every $i\in\set Z$, the monic minimal order annihilating operator of $\sigma^i(p)$ is
\[
  L^{(i)}:=S^s+\sigma^i(\ell_{s-1})S^{s-1}+\cdots+\sigma^i(\ell_0),
\]
and for every $u\in K\setminus\{0\}$, the monic minimal order annihilating operator of $\frac1uq$ is
\[
  \frac1{\sigma^s(u)}Mu=S^s+m_{s-1}\frac{\sigma^{s-1}(u)}{\sigma^s(u)}S^{s-1} + \cdots + m_0\frac{u}{\sigma^s(u)}.
\]
A necessary condition for a pair $(i,u)$ to be a solution to the shift equivalence problem is that $L^{(i)}=\frac1{\sigma^s(u)}Mu$.
Therefore, for every such pair we must have
\[
  \frac{\sigma^i(\ell_k)}{m_k} = \frac{\sigma^k(u)}{\sigma^s(u)}
\]
simultaneously for all~$k\in\set\{0,\dots,s\}$.

Observe that $L$ must have at least three terms. If it had only two terms, we would have $L=S^s+\ell_0$.
This means $\sigma^s(p) = -\ell_0 p$, and this is a contradiction to $p$ not being special.
We can therefore assume that $L$ has at least three terms.
We may further assume that the coefficient of $S^k$ in $L$ is nonzero if and only if the coefficient of $S^k$ in $M$
is nonzero, because if this is not the case, then the shift equivalence problem has no solution.

\begin{lemma}
  Under these circumstances, we have
  \begin{equation}\label{eq:elim}
  \sigma^i\Bigl(\frac{\ell_k/\sigma^s(\ell_k)}{\sigma^k(\ell_0)/\sigma^s(\ell_0)}\Bigr) = \frac{m_k/\sigma^s(m_k)}{\sigma^k(m_0)/\sigma^s(m_0)}
  \end{equation}
  for every $k\in\{1,\dots,s-1\}$ such that $\ell_k\neq0$.
\end{lemma}
\begin{proof}
From
\[
(a)\quad\frac{\sigma^i(\ell_0)}{m_0}=\frac{u}{\sigma^s(u)}
\quad\text{and}\quad
(b)\quad\frac{\sigma^i(\ell_k)}{m_k}=\frac{\sigma^k(u)}{\sigma^s(u)}
\]
we obtain
\[
 (c)\quad\frac{\sigma^i(\ell_0)}{m_0}\frac{m_k}{\sigma^i(\ell_k)} = \frac u{\sigma^k(u)}
\]
Apply $\sigma^k$ to $(a)$ and $\sigma^r$ to $(c)$ to obtain
\begin{alignat*}1
&(a')\quad\frac{\sigma^i(\sigma^k(\ell_0))}{\sigma^k(m_0)}=\frac{\sigma^k(u)}{\sigma^{s+k}(u)}
\quad\text{and}\\
&(c')\quad\frac{\sigma^i(\sigma^s(\ell_0))}{\sigma^s(m_0)}\frac{\sigma^s(m_k)}{\sigma^i(\sigma^s(\ell_k))} = \frac{\sigma^s(u)}{\sigma^{s+k}(u)}.
\end{alignat*}
Dividing $(a')$ by $(c')$ gives
\[
\frac{\sigma^i(\sigma^k(\ell_0))\sigma^s(m_0)\sigma^i(\sigma^s(\ell_k))}{\sigma^k(m_0)\sigma^i(\sigma^s(\ell_0))\sigma^s(m_k)}
=\frac{\sigma^k(u)}{\sigma^s(u)}.
\]
Finally, divide $(b)$ by this equation to obtain
\[
\frac{\sigma^k(m_0)\sigma^i(\sigma^s(\ell_0))\sigma^s(m_k)\sigma^i(\ell_k)}{\sigma^i(\sigma^k(\ell_0))\sigma^s(m_0)\sigma^i(\sigma^s(\ell_k))m_k}
=1.
\]
The claim follows from here.
\end{proof}

Unless both sides of Equation~\eqref{eq:elim} are constant, we get at most one candidate for $i$ and are done.
It remains to consider the case when both sides are constant for every $k$ with $\ell_k\neq0$ (and $m_k\neq0$).
In this case, the constant can only be~$1$, because $\ell_0,\ell_k,m_0,m_k$ are rational functions and $\sigma$
does not change leading terms.

If both sides of \eqref{eq:elim} are equal to $1$ then
\begin{alignat*}1
  \sigma^s(m_0)M\frac1{m_0}
  &=\sum_{k=0}^s \frac{m_k}{\sigma^k(m_0)/\sigma^s(m_0)}S^k
   =\sum_{k=0}^s \sigma^s(m_k)S^k = M^{(s)}.
\end{alignat*}
Therefore, if $i\in\set Z$ and $u\in K$ are such that $L^{(i)}=\frac1{\sigma^s(u)}Mu$, then
we also have $L^{(i+s)}=\frac1{\sigma^{2s}(u)}M^{(s)}\sigma^s(u)=\frac1{\sigma^s(\sigma^s(u)m_0)}M\sigma^s(u)m_0$.

This means that in terms of operators, the shift equivalence problem may have
more than one solution in the situation under consideration.
In the former cases, where there was at most one $i$
with $L^{(i)}=\frac1{\sigma^s(u)}Mu$, this $i$ is then the only candidate for which we can possibly
have $\sigma^i(p)=uq$. In the present situation, where there are infinitely many $i$'s that solve the
problem on the level of operators, it remains to determine which of them (if any) solves the original
problem in terms of $p$ and~$q$.

\begin{lemma}
  If $\sigma^s(m_0)M=M^{(s)}m_0$, then there is an operator $T\in C[S]$ with constant coefficients
  such that $M$ is a right factor of the symmetric product $T\otimes(S^s - m_0)$.
\end{lemma}
\begin{proof}
The condition $\sigma^s(m_0)M=M^{(s)}m_0$ means that for any solution $q$ of~$M$, also $\frac1{m_0}\sigma^s(q)$
is a solution of~$M$. But the solutions of $M$ form a $C$-vector space of dimension at most~$s$, so for every
solution $q$ of~$M$, the elements
\[
q,
\ \frac1{m_0}\sigma^s(q),
\ \frac1{m_0\sigma^s(m_0)}\sigma^{2s}(q),
\ \dots,
\ \biggl(\prod_{i=0}^{s-1}\frac1{\sigma^{si}(m_0)}\biggr)\sigma^{s^2}(q)
\]
are linearly dependent over~$C$.

Therefore, every solution of $M$ is also a solution of an operator of the form
\begin{alignat*}1
&S^{s^2} + c_{s(s-1)}\sigma^{s(s-1)}(m_0)S^{s(s-1)}
  + \cdots\\
&\cdots
+c_1\biggl(\prod_{i=1}^{s-1}\sigma^{si}(m_0)\biggr)S^s
+c_0\biggl(\prod_{i=0}^{s-1}\sigma^{si}(m_0)\biggr)
\end{alignat*}
for certain constants $c_0,c_s,\dots,c_{s(s-1)}$.

This operator can be factored as a symmetric product. Up to (irrelevant)
left-multiplication by an element of~$K$, it is equal to
\[
(S^{s^2} + c_{s(s-1)}S^{s(s-1)}+\cdots+c_1S^s+c_0)\otimes(S^s-m_0).
\]
This completes the proof.
\end{proof}

This means that every solution of~$M$, in particular~$q$, can be interpreted
as a product of a C-finite and an $s$-hypergeometric quantity, i.e., a quantity annihilated by an operator of the form $S^s-r$ for some rational function~$r$. (We do not claim
that these factors are elements of~$R$.)

We can reason analogously for $L$ and find that every solution of~$L$,
in particular~$p$, can be interpreted as a product of a C-finite and
an $s$-hypergeometric quantity, with the $s$-hypergeometric part annihilated by $S^s-\ell_0$.

If the $s$-hypergeometric factors are not also C-finite, then their comparison leads to at most
one candidate $i\in\set Z$ such that $\sigma^i(p)/q\in K$. In this comparison,
we must take into account that in the factorization of $p$ and $q$ into a C-finite
and an $s$-hypergeometric part, exponential terms $\lambda^x$ and polynomials in $x$ can be
freely moved from one factor to the other.

For the comparison, we use the Gosper-Petkov\v sek form~\cite{PWZbook1996} of $\ell_0$ and~$m_0$:
\[
\ell_0=\lambda\frac{\sigma(a)}{a}\frac{b}{c},\qquad
m_0=\tilde\lambda\frac{\sigma(\tilde a)}{\tilde a}\frac{\tilde b}{\tilde c}.
\]
We ignore $\lambda,\tilde\lambda,a,\tilde a$, as they correspond to the exponential and polynomial
part, respectively, and check if there is an $i\in\set Z$ such that $\sigma^{i}(b/c)=\tilde b/\tilde c$.
If so, then this $i$ is the only candidate for which we may have $\sigma^{i}(p)/q\in K$.

It remains to consider the case when $p$ and $q$ both are C-finite.
By Theorem 4.1 in \cite{stanley11}, there exist pairwise distinct $\lambda_1,\ldots,\lambda_t\in C$,
pairwise distinct $\mu_1,\ldots,\mu_{t'}\in C$ and polynomials $a_1, \ldots, a_t, b_1,\ldots,b_{t'}\in C[x]$ such that
\[
p=a_1(n)\lambda_1^n+\cdots +a_t(n)\lambda_t^n
\]
\[q=b_1(n)\mu_1^n+\cdots b_{t'}(n)\mu_{t'}^n,
\]
The requirement $\sigma^i(p)=uq$ translates into
\[
\sigma^i(a_1)\lambda_1^{i+n}+\cdots +\sigma(a_s)\lambda_s^{i+n}=ub_1\mu_1^n+\cdots+ub_t\mu_t^n.
\]
There is no solution unless $\{\lambda_1,\dots,\lambda_t\}=\{\mu_1,\dots,\mu_{t'}\}$, so we may assume that $t=t'$
and $\lambda_k=\mu_k$ for all $1\leq k\leq t$. Then we need
\[
\lambda_k^i\sigma^i(a_k)=ub_k
\]
for all~$k$. From any two such equations, say the $k$th and the $\ell$th, we get the constraint
\[
\left(\frac{\lambda_k}{\lambda_\ell}\right)^i\sigma^i\left(\frac{a_k}{a_\ell}\right)=\frac{b_k}{b_\ell}
\]
If $a_k/a_\ell$ is not a constant or $\lambda_k/\lambda_\ell$ is not a root of unity, then there is at most one solution~$i$.
If $a_k/a_\ell$ is a constant and $\lambda_k/\lambda_\ell$ is a root of unity for every choice of $k$ and~$\ell$,
then $p$ can be written as a product of $a_1$ and $\lambda_1^x$ and a $C$-linear combination of powers of roots of unity.
This implies that $p/a_1$ is a special polynomial, which conflicts with the assumption that $p$ is normal.

In conclusion, we have proven the correctness of the following algorithm.

\begin{algorithm}
  INPUT: $p,q\in R=C(x)[t_0,\dots,t_{r-1}]$ irreducible and normal\\
  OUTPUT: $i\in\set Z$ such that either $\sigma^i(p)/q\in C(x)$ or $p$ and $q$ are not shift-equivalent.

  \step11 Compute monic minimal annihilating operators $L,M\in C(x)[S]$ of $p$ and~$q$.
  \step21 If there is a $k$ such that the coefficient of $S^k$ is zero in one of the two operators
  but nonzero in the other, return~$0$.
  \step31 Let $s$ be the order of \,$L$ (and~$M$).
  \step41 For every $k\in\{1,\dots,s-1\}$ with $\ell_k\neq0$, do:
  \step52 If at least one of the two rational functions $a:=\frac{\ell_k/\sigma^s(\ell_k)}{\sigma^k(\ell_0)/\sigma^s(\ell_0)}$,
  $b:=\frac{m_k/\sigma^s(m_k)}{\sigma^k(m_0)/\sigma^s(m_0)}$ is not in $C$
  \step63 Return $i\in\set Z$ such that $\sigma^i(a)=b$, or $0$ if no such $i$ exists
  \step71 Compute the Gosper-Petkov\v sek form $\lambda\frac{\sigma(a)}{a}\frac{b}{c}$ of $\ell_0$ and
  the Gosper-Petkov\v sek form $\tilde\lambda\frac{\sigma(\tilde a)}{\tilde a}\frac{\tilde b}{\tilde c}$ of $m_0$.
  \step81 If $b/c\neq1$ or $\tilde b/\tilde c\neq1$ then
  \step92 Return $i\in\set Z$ such that $\sigma^i(b/c)=\tilde b/\tilde c$, or $0$ if no such $i$ exists
  \step{10}1 Compute the constants $\lambda_1,\dots,\lambda_s\in C$, polynomials $a_1,\dots,a_s, \\b_1,\ldots, b_s\in C[x]$, as above.
  \step{11}1 For $k=1,\dots,s$, do:
  \step{12}2 For $\ell=1,\dots,k-1$, do:
  \step{13}3 If $a_k/a_\ell$ is not a constant or $\lambda_k/\lambda_\ell$ is not a root of unity then
  \step{14}4 Return $i\in\set Z$ such that $(\frac{\lambda_k}{\lambda_\ell})^i\sigma^i(\frac{a_k}{a_\ell})=\frac{b_k}{b_\ell}$, or $0$ if no such $i$ exists.
\end{algorithm}


\section{The C-finite case}\label{sec:cfinite}
The problem of indefinite summation in the C-finite case has been investigated in~\cite{DuWei2024, HouWei2024, Wei2024} under specific assumptions. Let $A$ be defined as in \eqref{eqn:differencesystem} and assume that $A\in \GL_n(C)$. In \cite{HouWei2024}, the authors introduce a method for computing rational solutions of the equation $u\sigma(y)-vy=w$, where $u,v,w\in R$, under the assumption that $n=2$ and $A$ has two eigenvalues $\lambda_1,\lambda_2$ such that $\lambda_1/\lambda_2$ is not a root of unity. However, their method is not complete as they are unable to bound the multiplicities of irreducible special polynomials appearing in the denominator of solutions.
In \cite{DuWei2024}, assuming that $A$ is a diagonalizable matrix, the authors characterize all new constants in $F$ and present an algorithm for computing one rational solution of the equation $c\sigma(y)-y=f$, where $c\in C^*$ and $f\in F$. According to the proof of Proposition~\ref{prop:cfinitecase}, if $F$ contains no new constants, then $A$ will always be diagonalizable, and thus the C-finite case under our assumption is reduced to the case considered in \cite{DuWei2024}.

\medskip
\noindent {\bf Acknowledgement.}
The authors thank Prof. Michael Singer for many discussions about summation in finite terms and also thank the referees for their careful reading and their valuable suggestions. 


\bibliographystyle{plain}
\balance
\bibliography{integral}

\end{document}